\documentclass[10pt]{amsart}

\usepackage[utf8]{inputenc}
\usepackage{amsthm,amssymb,amsmath}
\usepackage{mathpazo}
\usepackage[linecolor=orange,backgroundcolor=orange!25,bordercolor=orange]{todonotes}
\usepackage{xargs,enumerate,graphicx,subcaption}
\usepackage[colorlinks]{hyperref}
\author{{\L}ukasz Garncarek}
\title[Factoriality of Hecke-von Neumann algebras]{Factoriality of Hecke-von Neumann algebras of right-angled Coxeter groups}


\address{University of Wroc{\l}aw, Institute of Mathematics,
  pl.~Grunwaldzki 2/4, 50-384 Wroc\-{\l}aw, Poland}

\email{Lukasz.Garncarek@math.uni.wroc.pl} 

\thanks{Research supported by the National Science Centre grant
  DEC-2012/05/N/ST1/02829. During part of the work on this paper the
  author was also supported by a scholarship of the Foundation for
  Polish Science. Part of the work on this paper was conducted during
  the author's internship at the Warsaw Center of Mathematics and
  Computer Science.}

\newtheorem{theorem}{Theorem}[section]
\newtheorem*{maintheorem}{Main Theorem}
\newtheorem{lemma}[theorem]{Lemma}
\newtheorem{proposition}[theorem]{Proposition}

\theoremstyle{definition}

\theoremstyle{remark}

\numberwithin{equation}{section}

\newcommand{\ZZ}{\mathbb{Z}}
\newcommand{\NN}{\mathbb{N}}
\newcommand{\CC}{\mathbb{C}}

\newcommand{\norm}[1]{\left\lVert#1\right\rVert}
\newcommand{\abs}[1]{\left\lvert#1\right\rvert}

\newcommand{\B}{\mathcal{B}}

\newcommand{\Nq}{\mathcal{N}_q}
\newcommand{\isom}{\cong}

\newcommand{\fp}{\star}

\begin{document}
\maketitle

\begin{abstract}
  The Hecke algebra $\CC_q[W]$ of a Coxter group $W$, associated to
  parameter $q$, can be completed to a von Neumann algebra
  $\Nq(W)$. We study such algebras in case where $W$ is
  right-angled. We determine the range of $q$ for which $\Nq(W)$ is a
  factor, i.e.\ has trivial center. Moreover, in case of nontrivial
  center, we prove a result allowing to decompose $\Nq(W)$ into a
  finite direct sum of factors.
\end{abstract}

\section{Introduction}
\label{sec:introduction}

For $q>0$ the Hecke algebra $\CC_q[W]$ of a Coxeter group $W$ is a
deformation of the group algebra of $W$, consisting of finitely supported
functions on $W$ with a modified product, yielding the ordinary group
algebra for $q=1$. In case where $q$ is an integer, the Hecke algebra
has a nice geometric interpretation. Recall that a building of type
$W$ can be thought of as a space endowed with a $W$-valued metric
\cite[Chapter 5]{Abramenko2008}. If it is locally finite, i.e.\ for
every $w$ there are finitely many points $y$ at distance $w$ from a
fixed point $x$, then for every $w\in W$ we may consider an operator
$A_w$ on the space of finitely supported functions on the building,
which averages the functions over spheres of radius $w$. If all
spheres of radius $s$, where $s$ is any of the standard generators of
$W$, have cardinality $q$, the operators $A_w$ generate an algebra
isomorphic to $\CC_q[W]$ \cite{Parkinson2006}. For $q=1$ the averages
are taken over $1$-element sets, and are in fact translations, so one
indeed gets the group algebra.

The Hecke algebra has a natural action on the Hilbert space of
square-integrable functions on $W$, obtained by extending its action
on itself by left multiplication. It therefore completes to a von
Neumann algebra---the Hecke-von Neumann algebra $\Nq(W)$. Its
importance stems from the theory of weighted $L^2$-cohomology of
Coxeter groups \cite{Davis2007}, where the cohomology spaces
are modules over $\Nq(W)$. 

A generalization of the Singer Conjecture \cite[Conjecture
14.7]{Davis2007} deals with vanishing of certain weighted cohomology
spaces of $W$. The  motivation for this work was to try to
approach this problem using the central decomposition of $\Nq(W)$ to
better understand the structure of these spaces. As it turns
out, although the centers of $\Nq(W)$ can be nontrivial, they
contribute nothing new in the subject of decomposing the weighted
cohomology of $W$; the decompositions described in \cite[Theorem
11.1]{Davis2007} are finer than those induced by the central
decomposition of $\Nq(W)$.

Although our results are not fit for the applications we initially had
in mind, they are interesting in their own right. Namely, we show that
the Hecke-von Neumann algebras of irreducible right-angled Coxeter
groups are factors, up to a $1$-dimensional direct summand. The main
result of the paper reads as follows.

\begin{maintheorem}[Theorem {\ref{thm:center-q-const}}]
    Suppose that $(W,S)$ is an irreducible right-angled Coxeter system
  with $\abs{S}\geq 3$. Then the Hecke-von Neumann algebra $\Nq(W)$
  is a factor if and only if
  \begin{equation}
    q \in [\rho, \rho^{-1}],
  \end{equation}
  where $\rho$ is the convergence radius of $W(t)$, the spherical
  growth series of $W$. Moreover, for $q$ outside this interval,
  $\Nq(W)$ is a direct sum of a factor and $\CC$.
\end{maintheorem}

Since taking centers commutes with tensor products, and the Hecke-von
Neumann algebra $\Nq(W)$ of an arbitrary right-angled Coxeter group
$W$ is the tensor product of the Hecke-von Neumann algebras of
irreducible factors of $W$, the central decomposition of $\Nq(W)$
follows in the general case.

The paper is divided into two parts. The first part consists of
Sections~\ref{sec:coxeter-groups-1} and \ref{sec:double-cosets}, and
is purely group-theoretic. Section~\ref{sec:coxeter-groups-1}
introduces the basic notions related to Coxeter groups, and in
Section~\ref{sec:double-cosets} we analyze certain double cosets in a
right-angled Coxeter group, and use them to define the graph
$\Gamma(W,S)$, whose connectivity is related to restrictions satisfied
by elements of the center of $\Nq(W)$. We show that for irreducible
$W$, this graph consists of a single connected component and at most
two isolated vertices.  The remaining two sections constitute the
analytic part. In Section~\ref{sec:hecke-algebras-1} we define the
Hecke-von Neumann algebras as algebras of operators on the Hilbert
space $\ell^2(W)$, the symbols of its elements in $\ell^2(W)$, and the
isomorphism $j$, which allows us to restrict to the case where $q \leq
1$. Section \ref{sec:symb-elem-cent} puts the graph $\Gamma(W,S)$ to
work, and characterizes the symbols of elements in the center of
$\Nq(W)$ for irreducible $W$. It turns out that, depending on $q$, the
potential symbol of a non-trivial element in the center has either
infinite $\ell^2$-norm, or defines a multiple of a unique
one-dimensional projection. This gives the two cases of the Main
Theorem. Finally, in Section~\ref{sec:final-remarks} we describe how
our results are related to the paper~\cite{Dykema} of Dykema, dealing
with free products of certain von Neumann algebras.

We would like to thank Adam Skalski, our mentor during the internship
at the Warsaw Center of Mathematics and Computer Science, for his
support and many stimulating discussions. We are also grateful to Jan
Dymara for introducing us to the subject of Hecke-von Neumann
algebras.


\section{Coxeter groups}
\label{sec:coxeter-groups-1}

A group $W$ generated by a finite set $S$, subject to the presentation
\begin{equation}
  \label{eq:def-cox-gp}
  W=\langle S \mid (st)^{m(s,t)}=1 : s,t\in S,\, m(s,t)\ne\infty\rangle,
\end{equation}
where the exponents $m(s,t)\in \{ 1,2,\ldots, \infty\}$ satisfy
$m(s,s)=1$, $m(s,t) \geq 2$ for $s\ne t$, and $m(s,t)=m(t,s)$, is
called a \emph{Coxeter group}. Together with the fixed generating set,
the pair $(W,S)$ is called a \emph{Coxeter system}. One usually
encodes the data from the definition of a Coxeter system into a graph,
called the \emph{Coxeter diagram}, having $S$ as the vertex set, and
an edge with label $m(s,t)$ between every two vertices $s,t$ with
$m(s,t) \geq 3$ (i.e.\ the edges between commuting generators are
omitted). One also considers the \emph{odd Coxeter diagram}, which
encodes information about conjugacy of generators, and is obtained 
from the Coxeter diagram by removing all the edges with even labels,
including $\infty$. Two generators $s,t\in S$ are conjugate in $W$ if
and only if they belong to the same connected component of the odd
Coxeter diagram of $(W,S)$ \cite[Lemma 3.3.3]{Davis2008}.

It is an important observation that for any $T\subseteq S$
the subgroup of $W$ generated by $T$ is also a Coxeter group, with the
same exponents as $W$ \cite[Theorems 4.1.6(i) and
3.4.2(i)]{Davis2008}. Such a subgroup is called
\emph{special}. The group $W$ is called \emph{irreducible} if it does
not decompose into a nontrivial direct product of special
subgroups. This is equivalent to the connectedness of its Coxeter diagram.

Any Coxeter group carries a \emph{word length}, given by
\begin{equation}
  \label{eq:word-length-def}
  \abs{w} = \min\{ n : (\exists s_1,\ldots,s_n\in S)\; w=s_1\cdots s_n\}.
\end{equation}
By \cite[Corollary 4.1.5]{Davis2008}, the word length in a special
subgroup agrees with the word length in $W$.  If $\abs{w}=n$ and
$w=s_1\cdots s_n$, then the word $s_1\cdots s_n$ is called a
\emph{reduced expression} for $w$. By \cite[Proposition
4.1.1]{Davis2008}, the set of generators appearing in any reduced
expression for $w$ is independent of the choice of expression. We will
denote it by $S(w)$. With this notation, the special subgroup $W_T$
consists exactly of the elements $w\in W$ with $S(w)\subseteq T$.

The power series
\begin{equation}
  W(t)=\sum_{w\in W} t^{\abs{w}} = \sum_{n=0}^\infty \abs{\{w\in W : \abs{w}=n\}}t^n
\end{equation}
is called the \emph{spherical growth series} of $W$. Its radius of
convergence will be denoted by $\rho$. If $W$ is infinite, then $\rho
\leq 1$. Since the coefficients of $W(t)$ are non-negative real
numbers, the series diverges at $t=\rho$.

The map $w\mapsto (-1)^{\abs{w}}$ is a group homomorphism, and hence
$\abs{sw}=\abs{w}\pm 1$ and $\abs{ws}=\abs{w}\pm 1$ for $w\in W$ and
$s\in S$. Among all groups generated by elements of order $2$, Coxeter
groups are characterized by the following equivalent conditions
\cite[Theorems 3.2.16 and 3.3.4]{Davis2008}. 
\begin{theorem}
  \label{thm:folding-exchange}
  Let $(W,S)$ be a Coxeter system, $w\in W$, and $s,s_1,\ldots,s_n,t\in S$.
  \begin{description}
  \item[Deletion condition] If $w=s_1\cdots s_n$ is not a reduced word then
    there exist $1\leq i < j \leq n$ such that $w =
    s_1\cdots\hat{s}_i\cdots\hat{s}_j\cdots s_n$.
  \item[Exchange condition] If $w=s_1\cdots s_n$ is a
    reduced expression and $sw$ is not, then $sw=s_1\cdots \hat s_i\cdots s_n$ for
    some $i$; a similar statement holds with $ws$ in place of $sw$.
  \item[Folding condition] If $sw$ and $wt$ are reduced words, then either
    $swt=w$ or $swt$ is reduced.
  \end{description}
\end{theorem}

A special case of a Coxeter system is a \emph{right-angled} Coxeter
system, in which the exponents $m(s,t)$ for $s\ne t$ are either $2$ or
$\infty$. They interpolate between the free and direct products of
copies of $\ZZ_2$. 
From now on, we will assume that $(W,S)$ is
right-angled. Its odd Coxeter diagram has no edges, and thus conjugate
generators are equal. This has consequences for the Deletion condition
(and the Exchange condition, which is a special case of the Deletion
condition). If we have a cancellation of the form $usvtw=uvw$ with
$u,v,w\in W$ and $s,t\in S$, then $v^{-1}sv=t$, so $s=t$ and $v$
commutes with $s$. By the next lemma, this means that $S(v)\subseteq
C(r)$, where $C(r)$ is the set of all generators commuting with
$r$. We will frequently use the following lemma, describing
centralizers of generators of a right-angled Coxeter group.

\begin{lemma}
  \label{thm:centralizer-racg}
  In a right-angled Coxeter system $(W,S)$ the centralizer of a generator
  $r\in S$ is the special subgroup $W_{C(r)}$. 
\end{lemma}
\begin{proof}
  The main theorem of \cite{Brink1996} gives the description of the
  centralizer of a generator $r\in S$ in an arbitrary Coxeter group as
  a semidirect product $W_T\rtimes F$, where $W_T$ is a special
  subgroup, and $F$ is the fundamental group of the connected
  component of the odd Coxeter diagram of $W$ containing $r$. The odd
  Coxeter diagram of a right-angled Coxeter group has no edges, so the
  centralizer is just $W_T$. Clearly, $T=C(r)$.
\end{proof}

\section{The graph $\Gamma(W,S)$}
\label{sec:double-cosets}

Let $(W,S)$ be an infinite irreducible right-angled Coxeter system,
and let $s,t\in S$ satisfy $m(s,t)=\infty$. Denote by $D$ the special
subgroup of $W$ generated by $s$ and $t$. It is isomorphic to
$\ZZ_2*\ZZ_2$. Reduced expressions for elements of $D$ are just
alternating words in $s$ and $t$, and in particular they are
unique. 

To any $w\in W$ there corresponds a double coset of $D$, defined as
$DwD=\{ dwd' : d,d'\in D\}$. By \cite[Lemma 4.3.1]{Davis2008}, it
contains a unique element $w_0$ of minimal length, such that any $u\in
DwD$ can be written as $dw_0d'$ with $d,d'\in D$ and
$\abs{u}=\abs{d}+\abs{w_0}+\abs{d'}$.

We will call the double coset $DwD$ \emph{non-degenerate} if its
shortest element $w_0$ does not centralize $D$, i.e.\ it does not
commute with either $s$ or $t$. We are ready to define the main tool
of this paper, the graph $\Gamma(W,S)$. The vertex set of
$\Gamma(W,S)$ is $W$, and for $w\in W$ and $s\in S$ there are edges
joining $w$ to both $ws$ and $sw$, provided that there exists $t\in S$
such that $m(s,t)=\infty$ and the double coset $DwD$ of $D=\langle
s,t\rangle$ is non-degenerate.  In Section~\ref{sec:hecke-algebras-1}
we will show that each edge in this graph corresponds to a restraint
on the symbol of an operator in the center of the Hecke-von Neumann
algebra. Now we will study connectivity of $\Gamma(W,S)$. We will show
that under some mild assumptions it consists of an isolated vertex $1$
and a single connected component.

For $w\in W$ its \emph{right} and \emph{left descent sets} are defined
as
\begin{equation}
  D_R(w) = \{ s\in S : \abs{ws} < \abs{w}\} \quad \text{and} 
  \quad D_L(w)= \{ s\in S : \abs{sw} < \abs{w}\}.
\end{equation}
The right (resp.\ left) descent set of $w$ is the set of all letters
in which a reduced expression can end (resp.\ begin). By the
discussion above, the generators in $D_R(w)$ pairwise commute, and
thus if $W$ is infinite, the descent sets are proper subsets of
$S$. Observe, that $\abs{vw}=\abs{v}+\abs{w}$ if and only if
$D_R(v)\cap D_L(w)=\emptyset$. Indeed, if the intersection is
nonempty, $v$ and $w$ have reduced expressions $\mathbf{v}$ and
$\mathbf{w}$, respectively ending and starting with the same
generator, which cancels out in their concatenation. If on the other
hand a cancellation of a pair of generators occurs in $\mathbf{vw}$,
then the subword of $\mathbf{vw}$ lying between the innermost pair of
canceling generators is reduced, and all its letters commute with the
deleted generator, so in fact it can be moved to the end and lies in
both $D_R(v)$ and $D_L(w)$. Yet another useful observation is the
following. Suppose that $\abs{ws}>\abs{w}$ for some $w\in W$ and $s\in
S$. In this case, we have
\begin{equation}
  \label{eq:descent-added-gen}
  D_R(ws) = (D_R(w)\cap C(s)) \cup \{s\},
\end{equation}
since the set on the right is contained in $D_R(ws)$, and if $t\in
D_R(ws)\setminus \{s\}$, then $t$ commutes with $s$ and lies in
$D_R(w)$.

\begin{lemma}
  \label{lem:regular-join}
  Suppose that $(W,S)$ is an irreducible infinite right-angled Coxeter system,
  and let $v,w\in W$. Then there exists $u\in W$ such that
  $\abs{vuw}=\abs{v}+\abs{u}+\abs{w}$.
\end{lemma}

\begin{proof}
  First, suppose that $D_L(w)\subseteq D_R(v)$. Since $W$ is infinite,
  the set $S\setminus D_R(v) = \{s_1,\ldots, s_n\}$ is nonempty. Set
  $v_k=vs_1\cdots s_k$ for $k=0,1,\ldots,n$. By
  \eqref{eq:descent-added-gen} we have $D_R(v_k) \subseteq D_R(v) \cup
  \{s_1,\ldots, s_k\}$, and thus $\abs{v_{k+1}} =
  \abs{v_k}+1$. Moreover,
  \begin{equation}
    D_R(v_{k+1})\cap D_R(v) = D_R(v_k)\cap C(s_{k+1}) \cap D_R(v) =
    \bigcap_{i=1}^{k+1}C(s_i)\cap D_R(v),
  \end{equation}
  and therefore
  \begin{equation}
    D_R(v_n)\cap D_L(w) \subseteq D_R(v_n)\cap D_R(v) = \bigcap_{s\not \in
      D_R(v)} C(s) \cap D_R(v) = \emptyset,
  \end{equation}
  since otherwise we would find a generator commuting with all the
  others, which contradicts irreducibility. Hence, we can use
  $u=s_1\cdots s_n$.

  Now, suppose that $D_L(w)\not\subseteq D_R(v)$. Let $D_L(w)\setminus
  D_R(v)=\{s_1,\ldots,s_n\}$ and denote $v_k=vs_1\cdots s_k$. As
  before, we conclude that $\abs{v_{k+1}}=\abs{v_k}+1$ and
  \begin{equation}
    D_R(v_{k+1})\cap D_L(w) = (D_R(v_k) \cap D_L(w)) \cup \{s_{k+1}\},
  \end{equation}
  since $D_L(w)\subseteq C(s_{k+1})$. Therefore, $D_L(w)\subseteq
  D_R(v_n)$, and $\abs{v_n}=\abs{v}+n$. We may now apply the first
  case with $v_n$ in place of $v$ to obtain the conclusion.
\end{proof}

\begin{lemma}
  \label{lem:all-gens-gamma}
  For an irreducible infinite right-angled Coxeter system $(W,S)$ with
  $\abs{S} \geq 3$ all elements $w\in W$ with $S(w)=S$ lie in the same
  component of\, $\Gamma(W,S)$.
\end{lemma}

\begin{proof}
  Let $v,w \in W$ satisfy $S(v)=S(w)=S$. Denote their reduced
  expressions by $\mathbf{v}$ and $\mathbf{w}$. By
  Lemma~\ref{lem:regular-join} there exists $u\in W$ with reduced
  expression $\mathbf{u}$ such that $\mathbf{vuw}$ is reduced. Let
  $\mathbf{uw}=s_1\cdots s_n$ with $s_i\in S$, and define
  $v_i=vs_1\cdots s_i$. We have $S(v_i)=S$, and by irreducibility,
  there exists a generator $t\in S$ not commuting with
  $s_{i+1}$. Moreover, for $D=\langle s_{i+1},t\rangle$ the double
  coset $Dv_iD$ is non-degenerate. Indeed, its shortest element
  contains all generators from $S\setminus\{s_{i+1},t\} \ne
  \emptyset$, and by irreducibility they can not all commute with
  $D$. Thus $v_i$ and $v_{i+1}$ are connected by an edge in
  $\Gamma(W,S)$. The same holds for the analogous path from $w$ to
  $vuw$. Therefore $v$ and $w$ lie in the same component.
\end{proof}

\begin{lemma}
  \label{lem:not-all-gens-path}
  Suppose that $(W,S)$ is an irreducible infinite right-angled Coxeter
  system, $w\in W\setminus\{1\}$ and $S(w)\ne S$. Then either
  \begin{enumerate}
  \item $W\cong \ZZ_2 * \ZZ_2^k$ and $w$ is the generator of the
    $\ZZ_2$ free factor, or
  \item there exists $s\in S\setminus S(w)$ such that $w$ is connected
    by an edge in $\Gamma(W,S)$ with $ws$ and $sw$.
  \end{enumerate}
\end{lemma}

\begin{proof}
  Suppose that situation (1) does not hold. We need to find $s\in
  S\setminus S(w)$ and $t\in S \setminus \{s\}$ such that
  $D=\langle s,t\rangle$ is infinite and the double coset $DwD$ is
  non-degenerate. We will consider three cases.

  \textbf{Case 1. The subgroup generated by $S\setminus S(w)$ is
    infinite.} We may thus find $s,t\in S\setminus S(w)$ generating
  infinite $D$. Since $s,t\not\in S(w)$, the shortest element in $DwD$
  is $w$. Assume that for all choices of $s$ and $t$ the double coset
  $DwD$ is degenerate, so that both $s$ and $t$ commute with all
  generators in $S(w)$.

  Denote by $T$ the subset of $S\setminus S(w)$ consisting of all
  generators not commuting with all of $S\setminus S(w)$. It is
  nonempty since $\langle S\setminus S(w)\rangle$ is infinite. By
  assumption, every $s\in T$ commutes with $S(w)$, and, by definition
  of $T$, with $S \setminus (S(w)\cup T)$. This gives a decomposition
  of $S$ into two non-empty commuting subsets $T$ and $S\setminus T$,
  which contradicts irreducibility of $(W,S)$.

  \textbf{Case 2. $S(w)$ generates an infinite subgroup of
  $W$.} It follows that there exist non-commuting $s,r\in S(w)$. If we
  can find $t\in S\setminus S(w)$ such that $D=\langle s, t\rangle$ is
  infinite, then the double coset $DwD$ is non-degenerate. Indeed, its
  shortest element $w_0$ is obtained from $w$ by canceling some
  occurrences of $s$ and $t$, and we still have $r\in S(w_0)$. Thus,
  $s$ does not commute with $w_0$.

  But if each $s\in S(w)$ which does not commute with all of $S(w)$,
  commutes with $S\setminus S(w)$, we get a contradiction similar to
  the one in Case 1.

  \textbf{Case 3. Both $S\setminus S(w)$ and $S(w)$ consist of
    pairwise commuting generators.} For $s\in S\setminus S(w)$ denote
  \begin{equation}
    N(s) = \{t\in S(w) : m(s,t)=\infty\}.
  \end{equation}
  If $N(s)$ contains at least two distinct elements $r,t$, then for
  $D=\langle s,t\rangle$ the shortest element of $DwD$ contains the
  generator $r$, and does not commute with $s$, so $DwD$ is
  non-degenerate. If $N(s)$ is empty, then $s$ commutes with all other
  generators, which contradicts irreducibility.

  What remains is the case where for all $s\in S\setminus S(w)$ we
  have $N(s)=\{f(s)\}$ for some function $f\colon S\setminus S(w) \to
  S(w)$. But for any $s\in S(w)$ all generators in the set $\{s\}\cup
  f^{-1}(s)$ commute with its complement. This contradicts
  irreducibility, provided that $S(w)\ne \{s\}$, which is excluded by
  assuming that the situation (1) from the statement of the Lemma does
  not hold.
\end{proof}

\begin{proposition}
  \label{prop:not-all-gens-gamma}
  Let $(W,S)$ be an irreducible infinite right-angled Coxeter
  system. Then all elements of $W$ except $1$ and, in case where
  $W\cong \ZZ_2*\ZZ_2^k$, the generator of the $\ZZ_2$ free factor,
  lie in the same connected component of $\Gamma(W,S)$.
\end{proposition}

\begin{proof}
We already know by Lemma \ref{lem:all-gens-gamma} that all elements
$w\in W$ with $S(w)=S$ lie in the same connected component of
$\Gamma(W,S)$. It remains to show that for $w$ as required in the
statement with $S(w)\ne S$ we may find a path in $\Gamma(W,S)$ from
$w$ to some element $v$ with $S(v)=S$. But this can be done by
inductively applying Lemma~\ref{lem:not-all-gens-path}.
\end{proof}

\section{Hecke-von Neumann algebras}
\label{sec:hecke-algebras-1}



Let $(W,S)$ be a Coxeter system, and let $q$ be a positive real
number. 
By \cite[Proposition 19.1.1]{Davis2008} there exists a unique
$*$-algebra $\CC_q[W]$ with basis $\{\tilde T_w : w\in W\}$, such
that for any $s\in S$ and $w \in W$
\begin{equation}
  \label{eq:def-hecke-product}
  \begin{split}
    \tilde T_s \tilde T_w & =
    \begin{cases}
      \tilde T_{sw} & \text{if $\abs{sw}>\abs{w}$,}\\
      q\tilde T_{sw} + (q-1)\tilde T_w & \text{otherwise};
    \end{cases}\\
  \tilde T_w^* & = \tilde T_{w^{-1}}.
\end{split}
\end{equation}
It is called the \emph{Hecke algebra} of $W$ associated to $q$. It
satisfies also the analogue of the first of formulas
\eqref{eq:def-hecke-product} with the order of $s$ and $w$ reversed.


If $q$ is an integer, this algebra can be interpreted as an algebra
generated by averaging operators on a certain space of functions
associated to a regular building of type $(W,S)$ and thickness $q$
(for details consult \cite{Parkinson2006}). When $q=1$, we get the
group algebra of $W$.


Denote $T_w = q^{-\abs{w}/2} \tilde T_w$. These normalized elements
satisfy for all $s\in S$ and $w\in W$ the identity
\begin{equation}
  \label{eq:normalized-hecke-product}
  T_sT_w =
  \begin{cases}
    T_{sw} & \text{if $\abs{sw}>\abs{w}$} \\
    T_{sw}+ p T_w & \text{otherwise}, 
  \end{cases}
\end{equation}
where
\begin{equation}
  \label{eq:def-r_s}
  p = \frac{q-1}{q^{1/2}}.
\end{equation}


The map $j\colon \CC_q[W]\to \CC_{q^{-1}}[W]$ defined as the linear
extension of
\begin{equation}
  j(T_w) = (-1)^{\abs{w}} T_w
\end{equation}
satisfies $j(T_sT_w)=j(T_s)j(T_w)$, and hence is an isomorphism of
$*$-algebras.


Now, consider the Hilbert space $\ell^2(W)$ of square-summable complex
functions on $W$. Using the embedding of $\CC_q[W]$ into $\ell^2(W)$
given by $T_w\mapsto \delta_w$ we may uniquely extend the left and
right actions of $\CC_q[W]$ on itself to bounded $*$-repre\-sen\-ta\-tions
on $\ell^2(W)$. We will identify $\CC_q[W]$ with the subspace of
$\B(\ell^2(W))$ consisting of left multiplication operators. Also, for
$T\in\CC_q[W]$, the corresponding right multiplication operator on
$\ell^2(W)$ will be denoted by $T^r$. The \emph{Hecke-von Neumann}
algebra $\Nq(W)$ is the von Neumann algebra generated by $\CC_q[W]$ in
$\B(\ell^2(W))$. The right-hand variant of $\Nq(W)$, generated by the
operators $T^r$ with $T\in\CC_q[W]$ will be denoted by $\Nq^r(W)$.  It
follows from \cite[Proposition 19.2.1]{Davis2008} that the commutant
of $\Nq$ is $\Nq^r$ and vice versa. The main object of study in this
paper is the center of the Hecke-von Neumann algebra $\Nq(W)$.

As in the case of the group von Neumann algebra, to any element
$T\in\Nq(W)$ we may associate its \emph{symbol}
$T\delta_1\in\ell^2(W)$. If the symbol of $T$ is $0$, then we have
\begin{equation}
  T\delta_w = TT_w^r\delta_1 = T_w^rT\delta_1 = 0,
\end{equation}
and thus $T=0$, so the mapping $T\mapsto T\delta_1$ is injective. If
$\xi\in\ell^2(W)$ is a symbol of an operator in $\Nq(W)$, it will be
denoted by $T(\xi)$. The same reasoning applies to $\Nq^r(W)$. In this
case the operator with symbol $\xi$ will be dented by
$T^r(\xi)$. 

Recall that there is an isomorphism $j:
\CC_q[W]\to\CC_{q^{-1}}[W]$. Through the natural embedding of
$\CC_q[W]$ it extends to an isometry of $\ell^2(W)$. Therefore, it
also extends to an isometric isomorphism of the Hecke-von Neumann
algebras $\Nq(W)$ and $\mathcal{N}_{q^{-1}}(W)$.

\section{Description of the center}
\label{sec:symb-elem-cent}

Let $(W,S)$ be a right-angled Coxeter system. Take $q>0$ and  let
\begin{equation}
  p=\frac{q-1}{q^{1/2}}.
\end{equation}
We may assume that $q\leq1$, since the case of
$q>1$ can be reduced to the latter using the isomorphism
$j$. Furthermore, we may assume that $(W,S)$ is irreducible. Indeed,
if there is a decomposition $(W,S)=(W_T\times W_U, T\cup U)$, then we
have a decomposition of the corresponding Hecke algebra into an
algebraic tensor product,
\begin{equation}
  \CC_q[W]\isom \CC_q[W_T]\otimes \CC_q[W_U],
\end{equation}
resulting in a tensor product decomposition of the Hecke-von Neumann
algebra
\begin{equation}
  \Nq(W)\isom\Nq(W_T)\otimes \Nq(W_U).
\end{equation}
By \cite[Corollary 11.2.17]{Kadison1986}
the center of a tensor product of von Neumann algebras is the tensor
product of their centers, and so the general case can be reduced to
the irreducible one.

We will
begin by describing the conditions satisfied by the symbol of an
operator in $\Nq(W)$ commuting with a single $T_s$.

\begin{lemma}
  \label{lem:commute-with-generator}
  Let $s\in S$ and $T(\xi)\in\Nq(W)$. Then $T(\xi)$ commutes with
  $T_s$ if and only if for all $w\in W$ such that
  $\abs{sws}=\abs{w}+2$ we have
  \begin{equation}\label{eq:commute-with-generator-conditions}
    \begin{split}
      \xi(sw) &=\xi(ws),\\
      \xi(sws)&=\xi(w)+p\xi(sw).
    \end{split}
  \end{equation}
\end{lemma}

\begin{proof}
  We start by assuming that $T(\xi)$ commutes with $T_s$. If
  $\abs{sws}=\abs{w}+2$, then $\abs{sw}=\abs{ws}=\abs{w}+1$ and we
  have
  \begin{equation}
    \label{eq:commute-with-gen-1}
    \begin{split}
      \xi(ws) & = \langle T(\xi)\delta_1, T^r_s\delta_w \rangle =
      \langle T^r_sT(\xi)\delta_1,\delta_w\rangle = \langle
      T(\xi)T^r_s\delta_1,\delta_w \rangle = \\
      & = \langle T(\xi)T_s\delta_1,\delta_w\rangle = \langle
      T_sT(\xi)\delta_1,\delta_w\rangle = \langle T(\xi)\delta_1,
      T_s\delta_w\rangle = \xi(sw)\\
    \end{split}
  \end{equation}
  and
  \begin{equation}
    \label{eq:commute-with-gen-2}
    \begin{split}
      \xi(sws) & = \langle T(\xi),T_sT_s^r\delta_w\rangle = \langle
      T(\xi)T_sT^r_s\delta_1, \delta_w\rangle = \\
      &= \langle T(\xi)(\delta_1 + p\delta_s),\delta_w\rangle =
      \xi(w) + p\xi(sw)
    \end{split}
  \end{equation}
  
  Now, assume that for any $w\in W$ with $\abs{sws}=\abs{w}+2$
  conditions \eqref{eq:commute-with-generator-conditions} hold. It
  suffices to show that $T_sT(\xi)$ and $T(\xi)T_s$ have the same symbol,
  i.e.\ that 
  \begin{equation}
    \label{eq:commute-with-gen-3}
    T_s\xi = T_sT(\xi)\delta_1 = T(\xi)T_s\delta_1 = T^r_s\xi.
  \end{equation}
  The group $W$ has a decomposition into a disjoint sum of double
  cosets $\langle s\rangle w \langle s \rangle$, which induces a
  decomposition of $\ell^2(W)$ into a direct sum of subspaces
  invariant under both $T_s$ and $T^r_s$. Therefore, we just need to
  check whether equation~\eqref{eq:commute-with-gen-3} holds for the
  restriction of $\xi$ to any such double coset.

  Let $\{w,sw,ws,sws\}$ be a double coset of $\langle s \rangle$. It
  has a unique shortest element, which we assume to be $w$, and its
  cardinality is either 4 or 2. In the latter case $sw=ws$ and
  $sws=w$. A direct calculation shows that in both cases the
  restriction of $\xi$ satisfies \eqref{eq:commute-with-gen-3}.
\end{proof}

Now, we find the restrictions on the symbol of an operator in $\Nq(W)$
commuting with two elements $T_s$ and $T_t$.

\begin{proposition}
  \label{thm:centralizer-of-d-infty}
  Let $(W,S)$ be a right-angled Coxeter system. Suppose that $q\leq 1$
  and $s,t\in S$ satisfy $m(s,t)=\infty$. Consider an operator
  $T(\xi)\in\Nq(W)$ which commutes with both $T_s$ and $T_t$. Then the
  restriction of its symbol $\xi$ to any non-degenerate double coset
  $DwD$ of $D=\langle s,t\rangle$ with shortest element $w$ is given
  by the formula
  \begin{equation}
    \label{eq:centr-elt-reg-dbl-coset}
    \xi(dwd') = \xi(w)q^{(\abs{dwd'}-\abs{w})/2}.
  \end{equation}
\end{proposition}

\begin{proof}
  There are two kinds of non-degenerate double cosets of $D$, those
  in which the shortest element does not commute with both generators
  of $D$, and those where it commutes with exactly one of them. We
  will consider these two cases separately.

  \begin{figure}[htb]
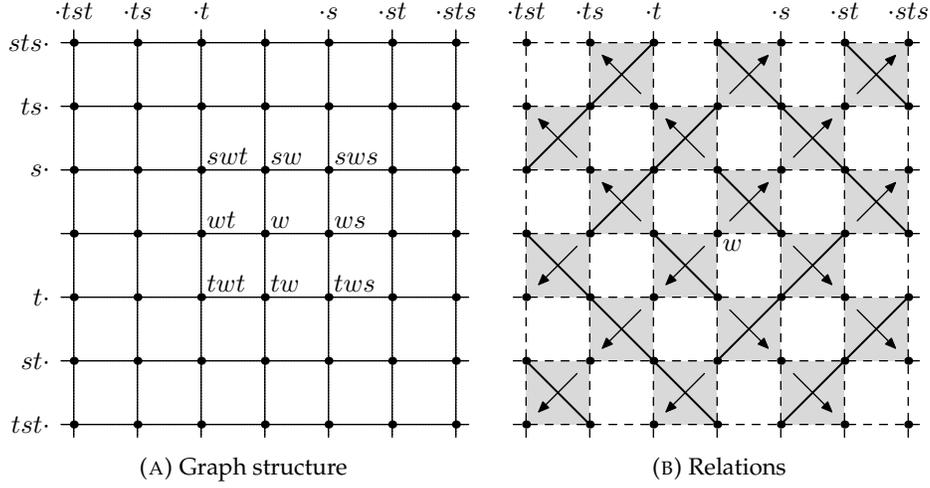

    \centering
    \begin{subfigure}{.5\linewidth}
      \centering
      \includegraphics{figs.1}
      \subcaption{Graph structure}
      \label{fig:graph-type-0}
    \end{subfigure}%
    \begin{subfigure}{.5\linewidth}
      \centering
      \includegraphics{figs.2}
      \subcaption{Relations}
      \label{fig:relations-case-0}
    \end{subfigure}
    \caption{Graph structure and relations between the values of the
      symbol $\xi$ of a central element of $\Nq(W)$ on a double coset
      $DwD$ for $w$ non-commuting with both generators of $D$}
  \end{figure}

  If $w$ does not commute with both generators of $D$, the mapping
  $D\times D\to DwD$ given by $(d,d')\mapsto d^{-1}wd'$ is
  bijective. Otherwise we would have $w=dwd'$ for some
  $(d,d')\ne(1,1)$, which is impossible, as no cancellation may occur
  in $dwd'$. This map induces the structure of the Cayley graph of $D\times
  D$ on $DwD$, whose fragment is presented in
  Fig.~\ref{fig:graph-type-0}.
  An important observation is that the path length of this graph
  faithfully reproduces word lengths of elements of $DwD$, by which we
  mean that for $v\in DwD$ its length is the sum of $\abs{w}$ and the
  distance between $v$ and $w$ in the graph.  In particular,
  Fig.~\ref{fig:relations-case-0} shows quadruples of elements of
  $DwD$ for which Lemma~\ref{lem:commute-with-generator} gives
  relations on the values of $\xi$. Every gray square corresponds to a
  single application of the Lemma to a quadruple $v,vr,rv,rvr$ with
  $v\in DwD$ and $r\in \{s,t\}$. The arrow is pointing from the
  shortest element $v$ to the longest one, $rvr$, and corresponds to
  the relation
  \begin{equation}
    \xi(rvr) = \xi(v)+p\xi(rv),
  \end{equation}
  while diagonal lines arising from the relations
  \begin{equation}
    \xi(rv)=\xi(vr)
  \end{equation}
  are joining elements on which the values of $\xi$ are equal.
  
  Now, take some $v_0\in DwD$ in which an arrow starts, and define
  inductively $v_{i+1}$ as the element at which the arrow starting in
  $v_i$ ends. Also, let $\lambda_i$ be the value of $\xi$ on the
  elements joined by the line crossing the arrow from $v_i$ to
  $v_{i+1}$. 
  \begin{figure}[htb]
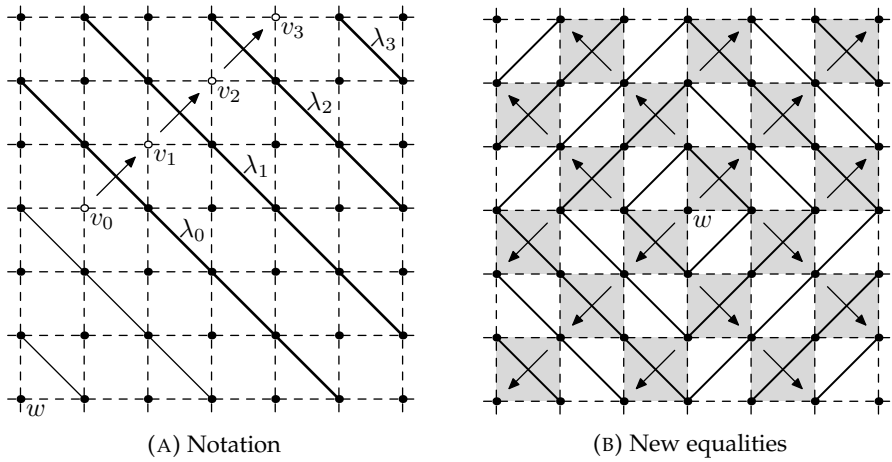

    \centering
    \begin{subfigure}{.5\linewidth}
      \centering
      \includegraphics{figs.3}
      \subcaption{Notation}
      \label{fig:path-in-graph}
    \end{subfigure}%
    \begin{subfigure}{.5\linewidth}
      \centering
      \includegraphics{figs.4}
      \subcaption{New equalities}
      \label{fig:more-relations}
    \end{subfigure}
    \caption{Additional relations for $w$ non-commuting with both
      generators of $D$.}
  \end{figure}
  This is illustrated in Fig.~\ref{fig:path-in-graph}.   
  We have, by passing to infinity,
  \begin{equation}
    \xi(v_0) = \xi(v_k) - p\sum_{i=0}^{k-1}\lambda_i = -p \sum_{i=0}^{\infty}\lambda_i.
  \end{equation}
  In particular, the value $\xi(v_0)$ depends only on the $\lambda_k$,
  which imposes new equalities, corresponding to adding to
  Fig.~\ref{fig:relations-case-0} diagonal lines as in
  Fig.~\ref{fig:more-relations}. It follows that $\xi(dwd')$ depends
  only on the distance between $dwd'$ and $w$ in the graph, i.e.\ on
  $\abs{dwd'}-\abs{w}$.

  Now, suppose that $w$ commutes with one of the generators of $D$,
  say $t$. The vertices of the graph in Fig. \ref{fig:graph-type-0}
  still correspond to elements of $DwD$, but no longer in a one-to-one
  manner. However, with the exception of the quarter containing $twt$
  (the lower left quarter in Fig.~\ref{fig:relations-case-1}),
  \begin{figure}[htb]
    \centering
    \includegraphics{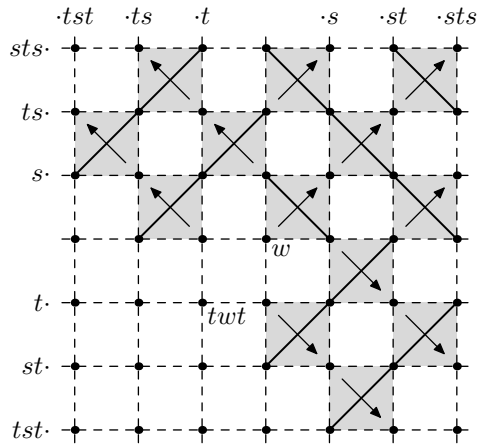}
    \caption{Relations between the values of the symbol $\xi$ of a
      central element of $\Nq(W)$ on a double coset $DwD$ for $w$
      commuting with one of the generators of $D$.}
    \label{fig:relations-case-1}
  \end{figure}
  the length of elements of the double coset is faithfully reflected
  by the path length in the graph. Indeed, if a pair of generators in
  $dwd'$ is canceled, one of them has to lie in $d$, while the other
  in $d'$. It follows that they must be equal to $t$, and be adjacent
  to $w$.

  As in the previous case we obtain a system of relations for values
  of $\xi$ on $DwD$, represented in Fig.~\ref{fig:relations-case-1},
  only this time distinct vertices may correspond to the same element
  of $DwD$, so there are additional equalities not indicated in the
  diagram. The same reasoning as before leads to the conclusion that
  $\xi(dwd')$ depends only on $\abs{dwd'}-\abs{w}$.

  We have therefore shown that for a non-degenerate double coset
  $DwD$, we have
  \begin{equation}
    \xi(dwd')=f(\abs{dwd'}-\abs{w})
  \end{equation}
  for some function $f\colon\NN\to \CC$. By
  Lemma~\ref{lem:commute-with-generator}, it satisfies the recurrence
  relation
  \begin{equation}
    f(n+2)=pf(n+1)+f(n).
  \end{equation}
  Its general solution is $f(n) = \alpha q^{n/2} + \beta
  (-q)^{-n/2}$. Since $q\leq 1$ and $\xi\in \ell^2(W)$, we have
  $\beta=0$, which ends the proof.
\end{proof}

We are now ready to describe the center of the Hecke-von Neumann
algebra of an irreducible right-angled Coxeter system with at least
$3$ generators. For $2$ generators we get the group $\ZZ_2*\ZZ_2$,
and it suffices to apply Lemma~\ref{lem:commute-with-generator}
to describe the large center of its Hecke-von Neumann algebra.

\begin{theorem}
  \label{thm:center-q-const}
  Suppose that $(W,S)$ is an irreducible right-angled Coxeter system
  with $\abs{S}\geq 3$. Then the Hecke-von Neumann algebra $\Nq(W)$
  is a factor if and only if
  \begin{equation}
    q \in [\rho, \rho^{-1}],
  \end{equation}
  where $\rho$ is the convergence radius of $W(t)$, the spherical
  growth series of $W$. Moreover, for $q$ outside this interval,
  $\Nq(W)$ is a direct sum of a factor and $\CC$.
\end{theorem}

\begin{proof}
  It is sufficient to consider the case where $q\leq 1$. Suppose that
  $T(\xi)\in Z(\Nq(W))$. If $v,w \in W$ are connected by an edge in
  the graph $\Gamma(W,S)$, they lie in the same non-degenerate double
  coset. Hence, by Proposition~\ref{thm:centralizer-of-d-infty},
  \begin{equation}
    \xi(w)q^{-\abs{w}/2} = \xi(v)q^{-\abs{v}/2}.
  \end{equation}
  By Proposition~\ref{prop:not-all-gens-gamma}, all nontrivial
  elements of $W$, except possibly one, lie in the same connected
  component $K$ of $\Gamma(W,S)$, and thus there exists
  $\lambda\in\CC$ such that for $w\in K$
  \begin{equation}
    \label{eq:formula-for-xi}
    \xi(w) = \lambda q^{\abs{w}/2}.
  \end{equation}
  If $K \ne W\setminus\{1\}$, the additional element outside $K$ is a
  generator $s\in S$. Since $\abs{S} \geq 3$ and $(W,S)$ is
  irreducible, there is another generator $t\in S$ not commuting with
  $s$, and by Lemma~\ref{lem:commute-with-generator},
  \begin{equation}
    \xi(s)=\xi(tst)-p\xi(ts)=\lambda q^{3/2} - (q-1)q^{-1/2}\lambda q
      = \lambda q^{1/2},
  \end{equation}
  so~\eqref{eq:formula-for-xi} extends to $W\setminus\{1\}$. It
  follows that $\xi$ is a linear combination of $\delta_1$ and $\zeta$
  defined by 
  \begin{equation}
    \zeta(w) = q^{\abs{w}/2}.
  \end{equation}
  But the $\ell^2$-norm of $\zeta$ can be expressed in terms of the
  growth series of $W$, namely
  \begin{equation}
    \norm{\zeta}^2_2 = \sum_{w\in W} q^{\abs{w}} = W(q).
  \end{equation}
  This is finite only for $q < \rho$, so for $q\geq \rho$ the
  Hecke-von Neumann algebra $\Nq$ is a factor.

  To finish the proof we will show that for $q<\rho$, the vector
  $\zeta$ is a symbol of an operator in $\Nq(W)$ which is proportional
  to a one-dimensional projection. This has been done in \cite[Lemma
  19.2.5]{Davis2008}, and we will reproduce the argument. First,
  observe that
  \begin{equation}
    \label{eq:generator-on-central-symbol}
    \begin{split}
      T_s^r\zeta = & \sum_{w\in W} q^{\abs{w}/2} T_s^r\delta_w =
      \sum_{\abs{w}<\abs{ws}} q^{\abs{w}/2}
      T_s^r(\delta_w+q^{1/2}\delta_{ws}) \\
      = & \sum_{\abs{w}<\abs{ws}} q^{\abs{w}/2} (\delta_{ws}+q^{1/2}\delta_w
      + pq^{1/2}\delta_{ws}) \\
      = & \sum_{\abs{w}<\abs{ws}} q^{(\abs{w}+1)/2}\delta_w +
      q^{\abs{w}/2+1} \delta_{ws} = q^{1/2}\zeta,
    \end{split}
  \end{equation}
  and therefore for finitely supported $\xi=\sum a_w\delta_w \in \ell^2(W)$ we obtain
  \begin{equation}
    \norm{T^r(\xi) \zeta}_2 = \norm{\sum a_wq^{\abs{w}/2}\zeta}_2 \leq
    \abs{\sum a_wq^{\abs{w}/2}}\norm{\zeta}_2 \leq W(q)\norm{\xi}_2\norm{\zeta}_2.
  \end{equation}
  This means that the map $\xi\mapsto T^r(\xi)\zeta$ extends to a
  bounded operator $Q$ on $\ell^2(W)$. Take $T^r(\eta)\in
  \Nq^r=\Nq'$. For finitely supported $\xi\in\ell^2(W)$ we have
  \begin{equation}
    T^r(\eta)Q\xi = T^r(\eta)T^r(\xi)\zeta = T^r(T^r(\eta)\xi)\zeta = QT^r(\eta)\xi,
  \end{equation}
  and thus $Q\in \Nq''=\Nq$. But the symbol of $Q$ is $\zeta$. It
  remains to observe that $T(\zeta)$ is self-adjoint and by
  \eqref{eq:generator-on-central-symbol}
  \begin{equation}
    T(\zeta)\zeta = \sum_{w\in W} q^{\abs{w}/2}T^r_w \zeta =
    \sum_{w\in W} q^{\abs{w}}\zeta = W(q)\zeta,
  \end{equation}
  so $W(q)^{-1}T(\zeta)$ is an orthogonal projection onto the space
  spanned by $\zeta$.
\end{proof}

\section{Final remarks}
\label{sec:final-remarks}

Adam Skalski brought to our attention the fact that the results of
this article are related to those of~\cite{Dykema}, where
decompositions of free products of certain von Neumann algebras are
studied.

Let $(M,\phi)$ be a pair consisting of a von Neumann algebra $A$ and a
faithful normal state $\phi$. Furthermore, let $\{ M_i\}$ be a family
of unital subalgebras of $M$. A \emph{reduced word} is then a product
$x_1\cdots x_k$ where $x_i\in M_{j(i)}$ and $j(i+1)\ne j(i)$ for all
$i < k$. The family $\{M_i\}$ is \emph{free} if any reduced word
$x_1\cdots x_k$ with $\phi(x_i)=0$ for all $i$ satisfies
$\phi(x_1\cdots x_k)=0$. Now, if $(M_i,\phi_i)$ are von Neumann algebras endowed with faithful
normal states, their free product is the unique von Neumann algebra
$M=\bigstar_i (M_i,\phi_i)$ with a faithful normal state $\phi$ and embeddings
$(M_i,\phi_i)\hookrightarrow (M,\phi)$ such that the embedded copies of $M_i$ are
free and together generate $M$. For more details on free products
consult \cite{Voiculescu2002}.

The Hecke-von Neumann algebra $\Nq(W)$ comes equipped with a faithful
normal state $\phi(T)=\langle T\delta_1, \delta_1 \rangle$ and we will
implicitly assume that all Hecke-von Neumann algebras are endowed with
this natural state. Also, when considering commutative algebras
$L^\infty(\Omega,\mu)$ for a measure space $(\Omega,\mu)$, we assume
that they are endowed with the integral with respect to $\mu$, which
is a state when $\mu(\Omega)=1$.

\begin{lemma}
  \label{lem:decomposition-into-free-product}
  Suppose that $W$ decomposes into a free product of special subgroups
  $W_i$. Then $\Nq(W)$ is the free product of the Hecke-von Neumann
  algebras $\Nq(W_i)$ associated to the free factors $W_i$ of $W$.
\end{lemma}

\begin{proof}
  It is enough to consider a free product $W=U*V$. Denote the
  states on the corresponding Hecke-von Neumann algebras by $\phi_W,
  \phi_U$, and $\phi_V$. Clearly, $\Nq(W)$ is generated by embedded
  copies of $\Nq(U)$ and $\Nq(V)$, on which $\phi_W$ restricts to
  $\phi_U$ and $\phi_V$, respectively.

  Now, a nontrivial reduced word $T=T_1\cdots T_k$ in elements of $\ker \phi_U$
  and $\ker \phi_V$ can be approximated by linear combinations of nontrivial
  reduced words of the form $T(w_1)\cdots T(w_k)$, where $w_i$ are
  nontrivial elements of $U$ and $V$. But such a word is equal to
  $T(w_1\cdots w_k)$, where $w_1\cdots w_k$ is a nontrivial reduced
  word in $W$, so $\phi_W(T)=0$.
\end{proof}

In \cite{Dykema}, among others, the free products of finite-dimensional
commutative von Neumann algebras were studied. We may thus apply these
results to Hecke-von Neumann algebras of right-angled Coxeter groups
which decompose into free products of abelian groups $\ZZ_2^k$. Let us
first study the Hecke-von Neumann algebras of these groups in more
detail.

\begin{lemma}
  \label{lem:hvn-z2}
  There exists an isomorphism of the Hecke-von Neumann algebra
  $\Nq(\ZZ_2^k)$ and the commutative algebra
  $L^\infty(\ZZ_2^k,\mu_k)$, where
  \begin{equation}\label{eq:commutative-hvn-measure}
    \mu_k(w)=\frac{q^{\abs{w}}}{(q+1)^k},
  \end{equation}
  preserving their natural states.
\end{lemma}

\begin{proof}
  We have $\Nq(\ZZ_2^k) \isom \Nq(\ZZ_2)^{\otimes k}$, so it is enough
  to consider the group $\ZZ_2$. Commutativity is clear, and $\dim
  \Nq(\ZZ_2)=2$, so we just need to exhibit a decomposition of the
  unit of $\Nq(\ZZ_2)$ into two projections. A simple calculation
  shows that this decomposition is
  \begin{equation}
    1 = \left(\frac{\sqrt{q}}{q+1} T_s + \frac{1}{q+1}\right) + \left( \frac{-\sqrt{q}}{q+1} T_s + \frac{q}{q+1}\right),
  \end{equation}
  and thus $\Nq(\ZZ_2)$ is isomorphic to $\CC^2$ with the state
  \begin{equation}
    (z_1,z_2)\mapsto \frac{z_1+qz_2}{1+q},
  \end{equation}
  i.e.\ to $L^\infty(\ZZ_2,\mu_1)$. Finally, we see that
  $L^\infty(\ZZ_2,\mu_1)^{\otimes k} \isom L^{\infty}(\ZZ_2^k, \mu_k)$.
\end{proof}

Now we will recall Theorem 2.3 and Proposition 2.4 of \cite{Dykema},
reformulated to better suit our usage. For $r>1$ by $L(F_r)$ we will
denote the interpolated free group factor endowed with its standard
state. Also, $L(F_1)$ will stand for the group von Neumann algebra
$L(\ZZ)$. Thus, we have $L(F_r) \fp L(F_s)=L(F_{r+s})$.

\begin{theorem}[\cite{Dykema}]
  \label{thm:dykema-free-products} 
  Let $X$ and $Y$ be finite (possibly empty) sets
  endowed with measures $\mu_X$ and $\mu_Y$ of full support, satisfying
  $\mu_X(X)\leq 1$ and $\mu_Y(Y) \leq 1$. Consider the von Neumann
  algebras $M$ and $N$ defined as follows
  \begin{enumerate}
  \item if $\mu_X(X)=\mu_Y(Y)=1$, the cardinalities
    $\abs{X},\abs{Y}\geq2$, and $\abs{X}+\abs{Y} \geq 5$, we set
    \begin{equation*}
      M = L^\infty(X,\mu_X) \qquad\text{and}\qquad N=L^\infty(Y,\mu_Y),
    \end{equation*}
  \item If $\mu_X(X)<1$ and $\mu_Y(Y)=1$, we set
    \begin{equation*}
      M = L(F_r)\oplus L^\infty(X,\mu_X) \qquad\text{and}\qquad N=L^\infty(Y,\mu_Y)
    \end{equation*}
    for some $r\geq1$, and endow $M$ with the state
    \begin{equation*}
      \phi_M(x,f) = (1-\mu_X(X))\phi_{L(F_r)}(x) + \int_X f\,d\mu_X,
    \end{equation*}
  \item if $\mu_X(X)<1$ and $\mu_Y(Y)<1$, we put
    \begin{equation*}
      M = L(F_r)\oplus L^\infty(X,\mu_X) \qquad\text{and}\qquad N= L(F_t)\oplus L^\infty(Y,\mu_Y)
    \end{equation*}
    for some $r,t\geq 1$, endow $M$ with the state $\phi_M$ defined above, and $N$ with the
    analogous state
    \begin{equation*}
      \phi_N(x,f) = (1-\mu_Y(Y))\phi_{L(F_r)}(x) + \int_Y f\,d\mu_Y.
    \end{equation*}
  \end{enumerate}
  With these definitions, we have
  \begin{equation}
    M\fp N \isom L(F_s)\oplus L^\infty(Z,\nu),
  \end{equation}
  where $s\geq 1$,
  \begin{equation}
    Z=\{(x,y)\in X\times Y : \mu_X(x)+\mu_Y(y) >1\},
  \end{equation}
  and $\nu(x,y)=\mu_X(x)+\mu_Y(y)-1$. 
\end{theorem}
In particular, under the assumptions of Theorem \ref{thm:dykema-free-products}, the free product $M\fp N$
is a factor if and only if
\begin{equation}
  \max_{x\in X}\mu_X(x) + \max_{y\in Y} \mu_Y(y) \leq 1.
\end{equation}
Now, consider the irreducible right-angled Coxeter group
\begin{equation}
  W=\ZZ_2^{k_1} * \ZZ_2^{k_2}*\cdots *\ZZ_2^{k_n}
\end{equation}
with $n\geq2$. If we
assume that $k_1 \geq 2$, we may inductively apply
Theorem~\ref{thm:dykema-free-products} to obtain the description of
$\Nq(W)$ as a direct sum of an interpolated free group factor and a
finite-dimensional commutative algebra:
\begin{equation}
  \Nq(W)\isom L(F_s) \oplus L^\infty(Z,\nu),
\end{equation}
where, using the notation of Lemma~\ref{lem:hvn-z2}, 
\begin{equation}
  Z = \left\{ (w_1,\ldots,w_n) \in \prod_{i=1}^n\ZZ_2^{k_i}
  : \sum_{i=1}^n\mu_{k_i}(w_i) > n -1\right\}.
\end{equation}
Suppose that $q\geq1$. We have
\begin{equation}
  \sum_{i=1}^n\mu_{k_i}(w_i) = \sum_{i=1}^n
  q^{\abs{w_i}-k_i} \left(\frac{q}{q+1} \right)^{k_i} \leq \frac{q}{q+1}\sum_{i=1}^n
  q^{\abs{w_i}-k_i},
\end{equation}
and if at least for one $i$ the inequality $\abs{w_i} < k_i$ holds,
this yields
\begin{equation}
   \sum_{i=1}^n\mu_{k_i}(w_i) \leq
   \frac{q}{q+1}\left(n-1+\frac{1}{q}\right) =
   n - 1 - \frac{n-2}{q+1} \leq n-1,
\end{equation}
since $n\geq2$. Therefore, depending on whether
\begin{equation}
  \sum_{i=1}^n \left(\frac{q}{q+1}\right)^{k_i} > n-1,
\end{equation}
the set $Z$ is either empty or contains exactly one element
$(w_1,\ldots, w_n)$, where $w_i$ is the unique longest element of
$\ZZ_2^{k_i}$. Calculating the exponential growth rate of $W$ by
considering the spherical growth series of its free factors would lead
to exactly the same condition. Hence,
Theorem~\ref{thm:dykema-free-products} allows to find the range of $q$
for which $\Nq(W)$ is a factor, additionally exhibiting it to be the
interpolated free group factor---but only in the case where $W$ is a
free product of finite right-angled Coxeter groups. Our result works
for arbitrary right-angled Coxeter groups, but gives less information
about the structure of the obtained factors.


\bibliographystyle{plain}
\bibliography{../../library}
\end{document}